\documentclass[10pt,a4paper,fleqn]{article}

\usepackage{graphicx}
\usepackage{mathptmx}      

\usepackage{amsmath}
\usepackage{amstext}
\usepackage{amsthm}
\usepackage{amsfonts}
\usepackage{amssymb}
\usepackage{stmaryrd}
\usepackage{mathtools}
\usepackage{authblk}
\usepackage{epsfig}
\usepackage{subfig}
\usepackage{inputenc}
\usepackage{rotating}
\usepackage{color}

\usepackage{lineno}
\usepackage{hyperref}
\modulolinenumbers[5]

\usepackage{titlesec}
\titleformat*{\section}{\large\bfseries}

\newtheoremstyle{MyTheorem}
{\topsep}       
{\topsep}       
{\itshape}  
{}          
{\bfseries}  
{.}         
{5pt plus 1pt minus 1pt}      
{}          

\theoremstyle{MyTheorem}
\newtheorem{theorem}{Theorem}[section]
\newtheorem{lemma}[theorem]{Lemma}
\newtheorem{cond}[theorem]{Condition}

\newcommand\IPst[0]{s.t.\;\;}
\newcommand\IPmin[0]{\min\;}

\newcommand{\ts}[0]{\textstyle\sum}

\DeclareMathOperator{\argmin}{argmin}
\newcommand\cref[1]{(\ref{#1})}

\usepackage{algorithm}
\usepackage[noend]{algpseudocode}

\begin{document}

\title{Compact Linearization for Binary Quadratic Problems subject to Assignment Constraints}

\author{Sven Mallach}

\affil{Institut f\"ur Informatik \authorcr Universit\"at zu K\"oln, 50923 K\"oln, Germany}
\date{\today}

\maketitle

\begin{abstract}
We prove new necessary and sufficient conditions to carry out a compact linearization
approach for a general class of binary quadratic problems subject to assignment constraints
as it has been proposed by Liberti in 2007. The new conditions resolve inconsistencies
that can occur when the original method is used. We also present a mixed-integer linear
program to compute a minimally-sized linearization. When all the assignment constraints have
non-overlapping variable support, this program is shown to have a totally unimodular
constraint matrix. Finally, we give a polynomial-time combinatorial algorithm that is
exact in this case and can still be used as a heuristic otherwise.
\end{abstract}

\section{Introduction}

In this paper, we are concerned with binary quadratic programs (BQPs)
that comprise some assignment constraints over a set of $\{0,1\}$-variables
$x_i$, $i \in N$ where ${N = \{ 1, \dots, n\}}$ for a positive integer $n$.
More precisely, let $K$ be a set
such that for each $k \in K$ there is some index set $A_k \subseteq N$ and
such that exactly one of the variables $x_i$, $i \in A_k$, must attain the value $1$.
We assume that $N \subseteq \{ A_k \mid k \in K\}$, i.e., the
set of problem variables is covered by the union of the sets $A_k$.
Additionally, bilinear terms $y_{ij} = x_i x_j$, $i,j \in N$, may occur in
the objective function as well as in the set
of constraints and are assumed to be collected in an ordered set
$E \subset V \times V$. By commutativity, there is no loss of generality
in requiring that $i \le j$ for each $(i,j) \in E$.
Assuming an arbitrary set of $m \ge 0$ further linear constraints
$C x + D y \ge b$ where $C \in \mathbb{R}^{m \times n}$ and
$D \in \mathbb{R}^{m \times |E|}$, the associated mixed-integer
program discussed so far can be stated as follows:
\begin{align}
\IPmin  &  c^Tx + b^Ty                                             &      &            && \nonumber \\
\IPst   &  \ts_{i \in A_k} x_{i}                                  &=\;   &  1         && \mbox{for all } k \in K     \label{bqp:assign} \\
        &  C x + D y                                           &\ge\; &  b         && \nonumber \\
	&  y_{ij}                                                  &=\;   & x_i x_j   && \mbox{for all } (i, j) \in E \label{bqp:bilinear} \\
	&  x_i                                                     &\in\; & \{0,1\}    && \mbox{for all } i \in N     \nonumber
\end{align}

The particular form studied here generalizes for example the quadratic assignment
problem which is known to be NP-hard \cite{SahniGonzales76} as are
BQPs with box constraints in general \cite{Sahni74}, although there are some exceptions \cite{Allemand2001}.
While there exist approaches to tackle BQPs directly, linearizations of quadratic and, more
generally, polynomial programming problems, enable the application of well-studied
mixed-integer linear programming techniques and have hence been an active field of research
since the 1960s. This is also in the focus of this paper where we concentrate on the question
how to realize constraints~\cref{bqp:bilinear} for this particular type of problem
by means of additional variables and additional linear constraints.

\emph{Related Work.}
The seminal idea to model binary conjunctions using additional \mbox{(binary)} variables
is attributed to Fortet~\cite{Fortet59,Fortet60} and discussed by Hammer and Ru\-dea\-nu~\cite{HammerRudeanu}.
This method, that is also proposed in succeeding works by Balas~\cite{Balas64}, Zangwill~\cite{Zangwill}
and Watters~\cite{Watters}, and further discussed by Glover and Woolsey~\cite{GloverWoolsey73}, requires two
inequalities per linearization variable. Only shortly thereafter, Glover and Woolsey~\cite{GloverWoolsey74}
found that the same effect can be achieved using \emph{continuous} linearization
variables when combining one of these two inequalities with two different ones.
The outcome is a method that is until today regarded as the \lq standard
linearization\rq\ and where, in the binary quadratic context, each product
$x_{i} x_{j}$ is modeled using a variable $y_{ij} \in [0,1]$ and three constraints:
\begin{align}
  y_{ij}  &\le\;  x_{i} && \label{pck:stdlin1} \\
  y_{ij}  &\le\;  x_{j} && \label{pck:stdlin2} \\
  y_{ij}  &\ge\;  x_{i} + x_{j} -1 && \label{pck:stdlin3}
\end{align}

Succeeding developments include a linearization technique without any additional
variables but using a family of (exponentially many) inequalities by Balas and Mazzola~\cite{BalasMazzola84}. 
Sherali and Adams showed how the introduction and subsequent linearization of additional
nonlinear constraints can be used to obtain tighter linear programming (LP) relaxations
for binary problems with (initially) linear constraints and a quadratic
objective function~\cite{SheraliAdams1986}. This
approach was later generalized in~\cite{SheraliAdamsRLT} to the so-called
reformulation-linearization-technique (RLT). A single application of the RLT to the bounds
constraints $0 \le x_i \le 1$ of a binary program leads exactly to the
above \lq standard linearization\rq\ as proposed by Glover and Woolsey in~\cite{GloverWoolsey74}.

Further linearization methods with more emphasis on problems where all
nonlinearities appear only in the objective function are by Glover~\cite{Glover75},
Oral and Kettani~\cite{OralKettani92a,OralKettani92b}, Chaovalitwongse et al.~\cite{CPP2004}, Sherali
and Smith~\cite{SheraliSmith2007}, Furini and Traversi~\cite{FuriniTraversi}, and, for general
integer variables, by Billionnet et al.~\cite{Billionnet2008}. Specialized
formulations for unconstrained binary
quadratic programming problems have been given by Gueye and Michelon~\cite{GueyeMichelon2009},
and Hansen and Meyer~\cite{HansenMeyer2009}.

For the particular binary quadratic problem as introduced at the beginning,
Liberti developed a very elegant alternative \emph{compact} linearization
approach that exploits the structure imposed by the assignment constraints~\cite{Liberti2007}.
It can be seen as a very special application of the RLT that first determines, for each
set $A_k$, $k \in K$, another set $B_k$ of original variables which are then multiplied
with the according assignment constraint related to $A_k$ yielding new additional equations.
The choice of the sets $B_k$ needs to be made such that the set of products $F$ introduced
this way covers the initial set of products $E$. Finally, the products in
$F$ are replaced by linearization variables. As already noted by
Liberti, conceptually, this approach is in line with and a generalization
of the strategy used by Frieze and Yadegar~\cite{FriezeYadegar83} for the
quadratic assignment problem.

\emph{Contribution.}
In this paper, we show that one of the conditions originally specified
as being necessary for the sets $B_k$ to hold in order to yield a correct compact
linearization is in fact neither necessary nor sufficient in every case. As a
consequence, when applying the original method to compute the sets $B_k$, inconsistent
value assignments to the set of created variables can result.
We reveal a new single necessary and sufficient condition to obtain a consistent linearization
and prove its correctness. As a positive side effect, this condition can lead to smaller
sets $B_k$ and hence to a smaller number of necessary additional variables and constraints.
In~\cite{Liberti2007}, also an integer program to compute minimum cardinality sets $B_k$
(and hence a minimum number of additional equations) has been given. We present a
similar mixed-integer linear program that establishes the new conditions and can now be
used to minimize both, the number of created additional variables and additional
constraints via a weighted objective function.
Moreover, we show that the constraint matrix of this program is totally unimodular if
all the sets $A_k$, $k \in K$, are pairwise disjoint. Additionally, we provide an exact
combinatorial and polynomial-time algorithm to compute optimal sets $B_k$ in this case.
With small modifications, the algorithm can also be used as a heuristic for the more
general problem with overlapping sets $A_k$.

\emph{Outline.}
In Sect.~\ref{s:CmpLin}, we review the compact linearization approach as developed
in~\cite{Liberti2007} and show that consistency of the linearization variables with
their associated original variables is not implied by the conditions specified.
New necessary conditions for a consistent linearization are characterized
in Sect.~\ref{s:RevCmpLin} together with a correctness proof. Further, a mixed-integer
linear program and a new combinatorial algorithm to compute compact linearizations
are given. We close this paper with a conclusion and final remarks in Sect.~\ref{s:Concl}.

\section{Compact Linearization for Binary Quadratic Problems}\label{s:CmpLin}
The compact linearization approach for binary quadratic
problems with assignment constraints by Liberti~\cite{Liberti2007}
is as follows:
For each index set $A_k$, there is a corresponding
index set $B_k \subseteq N$
such that for each $j \in B_k$ the assignment
constraint \cref{bqp:assign} w.r.t.~$A_k$ is multiplied with~$x_j$:
\begin{align}
    \sum_{i \in A_k} x_i x_j     &=\;  x_{j} && \mbox{for all } j \in B_k,\; \mbox{for all } k \in K \label{cmp:eqn0}
\end{align}

Each induced product $x_i x_j$ is then replaced by
a continuous linearization variable $y_{ij}$ (if $i \le j$) or $y_{ji}$ (otherwise).
We denote the set of such created bilinear terms
with $F$ and we may again assume that $i \le j$ holds for each
$(i,j) \in F$.
More formally, Liberti defined the set $F$~as
$$F = \{ \phi(i,j) \mid (i,j) \in \bigcup_{k \in K} A_k \times B_k \}$$
where $\phi(i,j) = (i,j)$ if $i \le j$ and $\phi(i,j) = (j,i)$ otherwise.
Using $F$, equations \cref{cmp:eqn0} can be
rewritten as follows:
\begin{align}
    \sum_{i \in A_k, (i,j) \in F} y_{ij}\ + \sum_{i \in A_k, (j,i) \in F} y_{ji}    &=\;  x_{j} && \mbox{for all } j \in B_k,\; \mbox{for all } k \in K \label{cmp:eqn1}
\end{align}

The choice of the sets $B_k$ is crucial for the correctness
and the size of the resulting linearized problem
formulation, as it directly determines the cardinality of
the set $F$ as well as the number of additional equations.
On the one hand, the sets $B_k$ must clearly be chosen such that
$F \supseteq E$.
On the other hand, this possibly (and in practice almost surely)
involves the creation of additional linearization variables for
some $i \in A_k$ and $j \in B_k$ where neither $(i,j) \in E$
nor $(j,i) \in E$.
Hence, the number of variables, $|F|$, will usually be larger than
$|E|$ but, as is discussed by Liberti and called
\emph{constraint-side compactness}, the number of equations can
be considerably smaller than $3 |E|$ as it would be with
the \lq standard\rq\ approach and the formulation will still be
at least as strong in terms of the LP relaxation of the~problem.

This latter property can, e.g., be shown by arguing that
solutions obeying all the equations~\cref{cmp:eqn1} also satisfy
the inequalities~\cref{pck:stdlin1}, \cref{pck:stdlin2}, and
\cref{pck:stdlin3} for all the variables introduced based on $F$.
Since the added equations are all equations of the original
problem multiplied by original variables, this 
also proves correctness of the linearization -- no solutions
feasible for the original problem can be excluded like this.
Liberti follows exactly this strategy.
His two main requirements for the choice of the sets is that
the covering conditions $E \subseteq F$ and $A_k \subseteq B_k$
for all $k \in K$ must be satisfied.

The condition $E \subseteq F$ 
and the definition of $F$ together imply that for each $(i,j) \in E$
there must be some $k \in K$ such that either $i \in A_k$ and $j \in B_k$ or
$j \in A_k$ and $i \in B_k$ which finally establishes that $(i,j) \in F$.
However, as we will show in the following, the
condition $A_k \subseteq B_k$ is not sufficient in every case in order
to ensure that $y_{ij} \le x_j$ and $y_{ij} \le x_i$ simultaneously hold
for all $(i,j) \in F$.

To see this, let $k \in K$ be such that $A_k \subsetneq B_k$. We can assume
without loss of generality that such a $k$ exists since otherwise $A_k = B_k$
must hold for all $k \in K$ and this implies that $E \subseteq F$ can only be
established if, for all $(i,j) \in E$, there is an $l \in K$ such that $i,j \in A_l$.
In this case, however, all bilinear terms could be resolved trivially as
$y_{ij} = 0$ for all $i \neq j$ and $y_{ij} = x_i$ for all
$i = j$. So let now $j \in B_k \setminus A_k$.
In our linearization system, we obtain for $j$ an equation:
\begin{align}
    \sum_{a \in A_k, (a,j) \in F} y_{aj}\ + \sum_{a \in A_k, (j,a) \in F} y_{ja}   &=\;  x_{j} && \label{eq:rhsj}
\end{align}

Now fix any particular $i = a \in A_k$ and assume, without loss of
generality, that $i < j$ ($i = j$ is impossible since $j \not\in A_k$).
Hence $(i,j) \in F$
and equation \cref{eq:rhsj} clearly establishes $y_{ij} \le x_j$.
The condition $A_k \subseteq B_k$ now requires
that there must be another equation for $i$ of the form:
\begin{align}
    \sum_{a \in A_k, (a,i) \in F} y_{ai}\ + \sum_{a \in A_k, (i,a) \in F} y_{ia}   &=\;  x_{i} && \label{eq:rhsi}
\end{align}

However, since $j \not\in A_k$, the variable $y_{ij}$
does not appear on the left hand side of \cref{eq:rhsi}. So if there
is no other $l \in K$, $l \neq k$, such that $i \in B_l$ and
$j \in A_l$, then there will be no equation that ever enforces
$y_{ij} \le x_i$.

The opposite case where, for some $(i,j) \in F$, $j \in A_k$ and
$i \in B_k \setminus A_k$ but there is no $l \in K$, $l \neq k$, such
that $j \in B_l$ and $i \in A_l$, leads to the converse problem that
there are equations that enforce
$y_{ij} \le x_i$ but none that enforce $y_{ij} \le x_j$.

Based on these observations, one can easily construct small examples
where it holds that $A_k \subseteq B_k$ for all $k \in K$, and $E \subseteq F$, but
there exist $(i,j) \in F$ for which
inconsistent value assignments to $y_{ij}$, $x_i$ and $x_j$ result.
This remains particularly true when the originally proposed integer
program
is used to determine and minimize the total
cardinality of the sets $B_k$.

\section{Revised Compact Linearization}\label{s:RevCmpLin}

In the previous section, we have seen that the condition
$A_k \subseteq B_k$ for all $k \in K$, is \emph{not sufficient}
in every case to enforce that the inequalities $y_{ij} \le x_i$
and $y_{ij} \le x_j$ are satisfied for all $(i,j) \in F$.
The discussion also already indicated that the two following
conditions are \emph{necessary} to enforce this.

\begin{cond} \label{cond:c1}
For each $(i,j) \in F$, there is a $k \in K$
such that $i \in A_k$ and $j \in B_k$.
\end{cond}
\begin{cond} \label{cond:c2}
For each $(i,j) \in F$, there is an $l \in K$
such that $j \in A_l$ and $i \in B_l$.
\end{cond}

For these two conditions, clearly, $k = l$ is a valid choice.
In this section, we will prove that these conditions
are also \emph{sufficient} in order to yield a correct linearization.
This also means that the inclusion $A_k \subseteq B_k$ is \emph{not}
a \emph{necessary} condition.

\begin{theorem}
Let $(i,j) \in F$. If Conditions~\ref{cond:c1} and~\ref{cond:c2}
are satisfied, then it holds that $y_{ij} \le x_i$, $y_{ij} \le x_j$ and
$y_{ij} \ge x_i + x_j - 1$.
\end{theorem}

\begin{proof}
By Condition~\ref{cond:c1}, there is a $k \in K$ such that
$i \in A_k$, $j \in B_k$ and hence the equation
\begin{align}
    \ts_{a \in A_k, (a,j) \in F} y_{aj} + \ts_{a \in A_k, (j,a) \in F} y_{ja}   &=\;  x_{j} && \tag{$*$} \label{eqn:rhsj_proof}
\end{align}
exists and has $y_{ij}$ on its left hand side. This establishes
$y_{ij} \le x_j$.

Similarly, by Condition~\ref{cond:c2}, there is an $l \in K$ such that
$j \in A_l$, $i \in B_l$ and hence the equation
\begin{align*}
    \ts_{a \in A_l, (a,i) \in F} y_{ai} + \ts_{a \in A_l, (i,a) \in F} y_{ia}   &=\;  x_{i} &&  \tag{$**$} \label{eqn:rhsi_proof}
\end{align*}
exists and has $y_{ij}$ on its left hand side. This establishes $y_{ij} \le x_i$.

As a consequence, $y_{ij} = 0$ whenever $x_i = 0$ or $x_j = 0$. In this case, the
inequality $y_{ij} \ge x_i + x_j - 1$ is trivially satisfied.
Now let $x_i = x_j = 1$. Then the right hand sides of both~(\ref{eqn:rhsj_proof}) and~(\ref{eqn:rhsi_proof})
are equal to $1$. The variable $y_{ij}$ (is the only one that) occurs on the
left hand sides of both of these equations. If $y_{ij} = 1$, this is consistent
and correct. So suppose that $y_{ij} < 1$ which implies that, in
equation~(\ref{eqn:rhsj_proof}), there is some $y_{aj}$ (or $y_{ja}$), $a \neq i$, with
$y_{aj} > 0$ ($y_{ja} > 0$).
Then, by the previous arguments and integrality of the $x$-variables, $x_a = 1$.
This is however a contradiction to the assumption that $x_i = 1$
as both $i$ and $a$ are contained in $A_k$.
\end{proof}

Minimum cardinality sets $B_k$ (steering the number of additional
equations) and minimum cardinality sets $F$ (steering the
number of additional variables) can be obtained using a mixed-integer
program. The following one is a modification of the integer program
in~\cite{Liberti2007} in order to implement
conditions~\ref{cond:c1} and~\ref{cond:c2}, but even more importantly, to enforce them not only for
$(i,j) \in E$ but for $(i,j) \in F$.
\begin{align}
    \IPmin  &  w_{eqn} \bigg( \sum_{k \in K} \sum_{1 \le i \le n} z_{ik} \bigg) + w_{var} \bigg( \sum_{1 \le i \le n}\sum_{i \le j \le n} f_{ij} \bigg)    \span \span  \span\span \nonumber \\
\IPst   &  f_{ij}                                         &=\;   & 1          && \mbox{for all } (i, j) \in E  \label{minB:EinF0}  \\
        &  f_{ij}                                         &\ge\; & z_{jk}     && \mbox{for all } k \in K, i \in A_k, j \in N, i \le j \label{minB:EinF3a} \\
	&  f_{ji}                                         &\ge\; & z_{jk}     && \mbox{for all } k \in K, i \in A_k, j \in N, j < i \label{minB:EinF3b} \\
        &  \sum_{k: i \in A_k} z_{jk} \qquad\qquad\qquad  &\ge\; & f_{ij}     && \mbox{for all } 1 \le i \le j \le n \label{minB:EinF1} \\
	&  \sum_{k: j \in A_k} z_{ik}                     &\ge\; & f_{ij}     && \mbox{for all } 1 \le i \le j \le n \label{minB:EinF2} \\
	&  f_{ij}                                         &\in\; & [0,1]    && \mbox{for all } 1 \le i \le j \le n \nonumber \\
        &  z_{ik}                                         &\in\; & \{0,1\}    && \mbox{for all } k \in K, 1 \le i \le n  \nonumber
\end{align}

The formulation involves binary variables $z_{ik}$ to be equal to $1$ if $i \in B_k$
and equal to zero otherwise. Further, to account for whether $(i,j) \in F$, there
is a (continuous) variable $f_{ij}$ for each $1 \le i \le j \le n$ that will be
equal to $1$ in this case and $0$ otherwise.
The constraints \cref{minB:EinF0} fix those $f_{ij}$ to $1$ where the corresponding
pair $(i,j)$ is contained in $E$. Further, whenever some $j \in N$
is assigned to some $B_k$, then we need the corresponding variables $(i,j) \in F$
or $(j,i) \in F$ for all $i \in A_k$ which is established by \cref{minB:EinF3a} and
\cref{minB:EinF3b}. Finally, if $(i,j) \in F$, then we require the two above
conditions to be satisfied, namely that there is a $k \in K$ such that $i \in A_k$ and $j \in B_k$ \cref{minB:EinF1}
and a (possibly different) $k \in K$ such that $j \in A_k$ and $i \in B_k$ \cref{minB:EinF2}.
Emphasis on either the number of created linearization variables or constraints can be given
using the weights $w_{var}$ and $w_{eqn}$ in the objective function.

The underlying problem to be solved is a special two-stage covering problem whose
complexity inherently depends on \emph{how} the sets
$A_k$, $k \in K$, cover the set $N$. In particular, if
$A_k \cap A_l = \emptyset$ for all
$k, l \in K$, $k \neq l$, then the choices to be made in constraints~\cref{minB:EinF1} and
\cref{minB:EinF2} are unique and the problem can be solved as a linear
program because the constraint matrix arising in this special case is totally
unimodular (TU). To show this, we make use of the following
lemma as stated in~\cite{NemWol}:
\begin{lemma}\label{lemma:TUnonzero}
If the $\{-1,0,1\}$-matrix $A$ has no more than two nonzero entries in each
column and if $\sum_i a_{ij} = 0$ whenever column $j$ has two nonzero coefficients, then $A$
is TU.
\end{lemma}

\begin{theorem}\label{thm:TU}
If $A_k \cup A_l = \emptyset$, for all $k,l \in K$, $l \neq k$, then the
constraint matrix of the above mixed-integer program is TU.
\end{theorem}

\begin{proof}
We interpret the constraint set \cref{minB:EinF3a}-\cref{minB:EinF2}
as the rows of a matrix $A = [ F\ Z ]$ where $F$ is the upper triangular matrix defined by
$\{ f_{ij} \mid 1 \le i \le j \le n \}$ and
$Z = (z_{ik})$ for $i \in N$ and $k \in K$.
Constraints \cref{minB:EinF3a} and \cref{minB:EinF3b} yield exactly one $1$-entry
in $F$, and exactly one $-1$-entry in $Z$.
Conversely, constraints \cref{minB:EinF1} and \cref{minB:EinF2} yield exactly one
$-1$-entry in $F$, and -- since $|\{k: i \in A_k\}| = 1$ for all $i \in N$ -- exactly
one $1$-entry in $Z$. Hence, each row $i$ of $A$ has exactly two nonzero entries and
$\sum_j a_{ij} = 0$. Moreover, the variable fixings~\cref{minB:EinF0} correspond to rows with
only a single nonzero entry or can equivalently be interpreted as removing columns from $A$
causing some of the other rows to have now less than two entries.
Thus, by Lemma~\ref{lemma:TUnonzero} and by the fact
that the transpose of any TU matrix is TU, $A$ is TU.
\end{proof}

Likewise, an exact solution can be obtained using a combinatorial algorithm
(listed as Algorithm~\ref{alg:construct}) where we assume that, for each
$i \in N$, the unique index $k: i \in A_k$ is given as $K(i)$.
Basically, an initial set $F_1 = E$ will then require some indices $i$
to be assigned to some unique
sets $B_k$. This may
lead to further necessary $y$-variables yielding a set $F_2 \supseteq F_1$ which
in turn possibly requires further unique extensions of the sets $B_k$
and so on until a steady state is reached.
The asymptotic running time of this algorithm can be bounded by $O(n^3)$.

\begin{algorithm}[h]
\begin{small}
\begin{algorithmic}[0]
\Function{ConstructSets}{Sets $E$, $K$ and $A_k$ for all $k \in K$}
\ForAll {$k \in K$}
\State $B_k \gets \emptyset$
\EndFor
\State $F \gets E$
\State $F_{new} \gets E$
\While{$\emptyset \neq F_{add} \gets \Call{Append}{F, F_{new}, K, A_k, B_k} $}
\State $F \gets F \cup F_{add}$
\State $F_{new} \gets F_{add}$
\EndWhile
\Return $F$ and $B_k$ for all $k \in K$
\EndFunction

\Function{Append}{Sets $F$, $F_{new}$, $K$, $A_k$ and $B_k$ for all $k \in K$}
\State $F_{add} \gets \emptyset$
\ForAll {$(i,j) \in F_{new}$}
\State $k^* \gets K(i)$
\State $l^* \gets K(j)$
\If {$j \not\in B_k$}
\State $B_k^* \gets B_k^* \cup \{j\}$
\ForAll {$a \in A_k^*$}
    \If {$a \le j$ and $(a,j) \not\in F$}
    \State $F_{add} \gets F_{add} \cup \{ (a, j) \}$
    \ElsIf {$a > j$ and $(j,a) \not\in F$}
    \State $F_{add} \gets F_{add} \cup \{ (j, a) \}$
    \EndIf
\EndFor
\EndIf
\If {$i \not\in B_l^*$}
\State $B_l^* \gets B_l^* \cup \{i\}$
\ForAll {$a \in A_l^*$}
    \If {$a \le i$ and $(a,i) \not\in F$}
    \State $F_{add} \gets F_{add} \cup \{ (a, i) \}$
    \ElsIf {$a > i$ and $(i,a) \not\in F$}
    \State $F_{add} \gets F_{add} \cup \{ (i, a) \}$
    \EndIf
\EndFor
\EndIf
\EndFor
\Return $F_{add}$
\EndFunction
\end{algorithmic}
\end{small}
\caption{A Simple Algorithm to construct $F$ and the sets $B_k$ for all $k \in K$.}
\label{alg:construct}
\end{algorithm}

For the more general setting with overlapping $A_k$-sets and hence
$|\{k: i \in A_k\}| \ge 1$ for $i \in N$, the above proof of Theorem~\ref{thm:TU}
fails and indeed, one can construct small artificial instances that have nonintegral
optima. Nonetheless, the combinatorial algorithm can still be
used when equipped with a routine to determine the indices $k^*$ and $l^*$.
One can follow, e.g., the following heuristic idea:
To ease notation, let $a(i, k) = 1$ if $i \in A_k$ and $a(i,k) = 0$ otherwise.
Similarly, let $b(i, k) = 1$ if $i \in B_k$ and $b(i,k) = 0$ otherwise.
Adding some $i$ to some $B_k$ for the first time involves potentially creating additional
variables $(u,i)$ or $(i,u) \in F$ for $u \in A_k$. More precisely, whether such a variable
must be newly created depends on whether or not there already is some $l \in K$ where
$u \in A_l$ and $i \in B_l$ or ${i \in A_l}$ and ${u \in B_l}$.
Hence, the number of necessarily created variables when adding $i$ to $B_k$ is:
$c(i,k) = \sum_{u \in A_k} ( \min_{l \in K} 1 - \max \{a(u,l) b(i,l), a(i,l) b(u,l)\})$.
So for a pair $(i,j) \in F$, we select $k^* = \argmin_{k: i \in A_k} c(j,k)$ 
and $l^* = \argmin_{k: j \in A_k} c(i,k)$. 
Each such choice is a locally best one that neither respects the interdependences with any other choices
nor the implications of the corresponding potentially induced new pairs $(a, j)$ (or $(j,a)$)
for $a \in A_k^*$ and $(a, i)$ (or $(i,a)$) for $a \in A_l^*$.

\section{Conclusion and Final Remarks}\label{s:Concl}

In this paper, we introduced new necessary and sufficient
conditions to apply the compact linearization approach for
binary quadratic problems subject to assignment constraints
as proposed by Liberti in~\cite{Liberti2007}. These conditions are proven to lead
to consistent value assignments for all linearization variables
introduced. Further, a mixed-integer program has been presented
that can be used to compute a linearization with the minimum number
of additional variables and constraints. We also showed that,
in the case where all the assignment constraints have non-overlapping
variable support, this problem can be solved as a linear program as
its constraint matrix is totally unimodular. Alternatively, an exact
polynomial-time combinatorial algorithm is proposed that can also
be used in a heuristic fashion for the more general setting with
overlapping variable sets in the assignment constraints.

\bibliographystyle{abbrv}
\bibliography{bibfile}

\end{document}